\documentclass[8pt,B5]{article}

\usepackage{graphicx,amsmath,latexsym,amssymb,amsthm,fancyhdr}

\usepackage{enumitem}
\usepackage{stmaryrd}

\font\chuto=cmbx10 at 16pt  

\newtheorem{thm}{Theorem}[section]
\newtheorem{cor}[thm]{Corollary}
\newtheorem{lem}[thm]{Lemma}
\newtheorem{prop}[thm]{Proposition}
\theoremstyle{definition}

\newtheorem{rem}[thm]{Remark}

\numberwithin{equation}{section}
\def\mydil{S}
\def\myxi{\xi}
\def\trp#1{#1^{T}}
\def\rnhat{\widehat{\R^n}}
\def\transp{T}
\def\R{\mathbb{R}}

\begin{document}

\thispagestyle{empty}

\centerline {\bf \chuto  Extensions of the Heisenberg Group by}
\medskip
\centerline {\bf \chuto  Two-Parameter Groups of Dilations}

\renewcommand{\thefootnote}{\fnsymbol{footnote}}

\vskip.8cm

%\centerline {Name1 Surname1$^{1,}${\footnote{\textit{Correspondence to: e-mail: \texttt{xxx@hotmail.com}}}},
\centerline {Eckart Schulz$^{1}$ and
             Adisak Seesanea$^{2,}${\footnote{\textit{The author is supported by Development and Promotion of Science and Technology Talents Project (DPST).}}}             
}

\renewcommand{\thefootnote}{\arabic{footnote}}

\vskip.5cm

\centerline{School of Mathematics, Suranaree University of Technology, Nakhon Ratchasima 30000, Thailand}
\centerline{\texttt{$^1$eckart@math.sut.ac.th ~ $^2$adisak@math.sut.ac.th}}

\vskip .5cm

%%%%%%%%%%%%%%%%%%%%%%%%%%%%%%%%%%%%%%%%%%%%%%%%%%%%%%%%%%%%%%%%%%%%%%%%%%%%%%%%%%%%
%%%%%%%%%%%%%%%%%%%%%%%%%%%%%%%%%%%%%%%%%%%%%%%%%%%%%%%%%%%%%%%%%%%%%%%%%%%%%%%%%%%%

\begin{abstract}
We introduce extensions of the multidimensional Heisenberg group $\mathbb{H}^n$  by two-parameter groups of dilations, and then classify the extended groups up to isomorphism, by employing Lie algebra techniques. We show that the groups are isomorphic to subgroups of the symplectic group $\textit{Sp(}n+1,\mathbb{R})$  as well as subgroups of the affine group $\textit{Aff}(n+1,\mathbb{R})$. Thus, they possess both, a metaplectic and a wavelet representation. Moreover, the metaplectic representation splits into a sum of two  subrepresentations which both are equivalent to the same subrepresentation of the wavelet representation.
\end{abstract}

%%%%%%%%%%%%%%%%%%%%%%%%%%%%%%%%%%%%%%%%%%%%%%%%%%%%%%%%%%%%%%%%%%%%%%%%%%%%%%%%%%%%
%%%%%%%%%%%%%%%%%%%%%%%%%%%%%%%%%%%%%%%%%%%%%%%%%%%%%%%%%%%%%%%%%%%%%%%%%%%%%%%%%%%%

\medskip



%%%%%%%%%%%%%%%%%%%%%%%%%%%%%%%%%%%%%%%%%%%%%%%%%%%%%%%%%%%%%%%%%%%%%%%%%%%%%%%%%%%%
%%%%%%%%%%%%%%%%%%%%%%%%%%%%%%%%%%%%%%%%%%%%%%%%%%%%%%%%%%%%%%%%%%%%%%%%%%%%%%%%%%%%
%%%%%%%%%%%%%%%%%%%%%%%%%%%%%%%%%%%%%%%%%%%%%%%%%%%%%%%%%%%%%%%%%%%%%%%%%%%%%%%%%%%%
%%%%%%%%%%%%%%%%%%%%%%%%%%%%%%%%%%%%%%%%%%%%%%%%%%%%%%%%%%%%%%%%%%%%%%%%%%%%%%%%%%%%

%%%%%%%%%% 

\section{Introduction}
The Heisenberg group $\mathbb{H}^n$ plays a fundamental role in quantum mechanics \cite{Folland} and signal processing \cite{Grochenig}. 
Recall that the matrix
$ \mathcal{J} = \left[ {\begin{smallmatrix}
  						 0 & -I_n  \\ I_n & 0  
   			 \end{smallmatrix} } \right ]
$ 
determines a \emph{symplectic form} on $\R^{2n}$ by
\begin{equation}\label{eq:sympl}
		\qquad\qquad\qquad 	\llbracket  w  , \tilde w \rrbracket =  \trp{w} \mathcal{J} \tilde  w 
					\qquad\qquad (w , \tilde w \in \R^{2n}).
\end{equation}
The \emph{Heisenberg group} is the set
\[
	\mathbb{H}^n  = \bigl\{  (  w ,z) \ :   \	 w  \in \R^{2n},\  z \in \mathbb{R} \,  \bigr \}
\]
endowed with the topology of $\R^{2n+1}$ and the group operation
\[
   ( w ,z)( \tilde w ,\tilde z) = \bigl  ( w  + \tilde w, z+ \tilde z + \frac{1}{2} \llbracket w , \tilde w \rrbracket \bigr ) .
\]
The \emph{symplectic group} $ Sp(n, \mathbb{R}) $ acts naturally on $\mathbb{H}^n$ by automorphisms.
 Recall here that the symplectic group is the set of all 
invertible matrices preserving the symplectic form (\ref{eq:sympl}),
\[
	 Sp(n, \mathbb{R}) = \left \{ \, \mathcal{A} \in GL_{2n}(\R) \ \big | \ 
	 	\llbracket   {\mathcal{A}  w ,\mathcal{A}  \tilde w  } \rrbracket
	 			= \llbracket  w ,  \tilde w  \rrbracket \qquad \forall \  w, \tilde w  \in {\R}^{2n} \ \right \},
\]
and each of its elements defines an automorphism $\alpha_{\mathcal{A}}$ of $\mathbb{H}^n$ by
\[		\alpha_{\mathcal{A}}(w,z) = (	\mathcal{A}w,z)	\]
 fixing the elements of the center $Z=\{(0,z):z \in \R \}$ of $\mathbb{H}^n$.

Our interest centers around  a different family of automorphisms. Separating the phase space $W=\{(w,0): w \in \R^{2n} \}$
into its two components $X=  \{  ((x,0),0   ) : x \in \R^n   \}$ and  $Y  = \{  ((0,y),0 ) : y \in \R^n   \}$,
we consider automorphisms which leave all three subgroups $X$, $Y$ and $Z$ invariant,
without fixing elements of the center $Z$.

For this purpose, it will be more convenient to work with the \emph{polarized Heisenberg group} $\mathbb{H}_{pol}^n $,
\[
	\mathbb{H}_{pol}^n  = \bigl\{ \, (  x , y,z) \ : 	 \ 	 x, y \in \R^n,\,  z \in \mathbb{R} \,  \bigr \}
\]
which has the group operation
\[
   ( x , y ,z)( \tilde x,\tilde y,\tilde z) = \left ( x + \tilde x,  y + \tilde y, z+\tilde z + \trp{y} \tilde x \right ) 
\]
and the simple representation as a matrix group
\[	\mathbb{H}_{pol}^n = \left \{  h( x , y, z) \, = \, 
\begin{bmatrix}
  		 1 & \trp{y} & z  \\
   		 0 & I_n  &  x   \\
  		 0 & 0 & 1
	\end{bmatrix}\  : \  x, y  \in \R^n,  \ z \in \R  \right \} \subset    GL_{n + 2} (\mathbb{R}).
\]
The two Heisenberg groups are isomorphic  via the map
\[	\qquad\qquad
	\Psi : (  w,z) \in \mathbb{H}^n \mapsto  h(  x , y,z+\textstyle \frac{1}{2} \trp{y}  x) \in \mathbb{H}_{pol}^n
				\qquad   \left ( w = \left [ \begin{smallmatrix} x\\ y  \end{smallmatrix} \right ],  \ x,y \in \R^n \right ).
\]

Now consider the closed subgroup of $GL_{n+2}(\R)$ of the form
\[
	D_0 = \left \{ \ 
						d(a,A,c) : =\left[ {\begin{array}{*{20}c}
   											a & 0 & 0  \\
  										 0 & {A} & 0  \\
   										0 & 0 & c  \\
 									\end{array} } \right]
 						\ : a,c \in \R \backslash \{0\}, \ A \in GL_n(\R) \, \right \}.
\]
There is a natural action $\alpha$ of $D_0$ on $\mathbb{H}_{pol}^n$ by conjugation,
\begin{equation}\label{equa:01}
				\qquad 	\alpha_d \left ( \,  h( x, y ,z) \, \right )  =
							d(a,A,c) \, h( x, y ,z) \, d(a,A,c)^{-1} \qquad (d=d(a,A,c) \in D_0).
\end{equation}
Given a closed subgroup $D$ of $D_0$, one thus obtains a semi-direct product 
$
				\mathbb{H}^{n}_{pol} \rtimes D
$
which can be represented as a matrix group 
\begin{equation}\label{eq:01}
				\mathbb{H}^{n}_{pol} \rtimes D \cong	\left\{ \, h(x, y ,z)d(a,A,c) \,:\, h(x, y, z) \in \mathbb H_{pol}^n,  \,
								d(a,A, c) \in D \, \right\} \subset GL_{n+2}(\R).
\end{equation}
Observe that replacing $d(a,A,c)$ with $d(c^{-1}a, c^{-1}A,1)$ results in the same automorphism and hence an isomorphic semi-direct product,
so one need only consider groups $D$ with $c=1$.
When $D=\{ \, d(1,A,1): A \in GL_n(\R) \, \} \cong GL_n(\R)$, this semi-direct product is called the 
\emph{affine-Weyl-Heisenberg group}. It has been extensively studied in \cite{Ali, Hogan, Kalisa, Torresani}.

In \cite{Schulz}, the groups (\ref{eq:01}) were studied in case $n=1$ (so that $A$ is a scalar) and $a,A \in \mathbb{R}^{+},$ $c=1$,
and were classified up to isomorphism. It was further noticed that  they are isomorphic to
 affine semi-direct products, 
\begin{equation}\label{equa:03}
		\mathbb{H}^{1}_{pol} \rtimes D \cong \mathbb R^{2} \rtimes H
\end{equation}
where $H$ is a closed subgroup of $GL_2(\R)$, and $\mathbb R^{2}$  and $H$ are identified with the groups of  matrices, 
\[
		\mathbb R^{2} \cong  \left\{  \, h(x,0,z) \; : \;  x  , z \in \mathbb{R} \,  \right\}
\]
and 
\[
		H \cong  \left\{ \, h(0,y,0)d(a,A,1) 	\,:\, a,A \in \R^+ \, \right \}.
\]
Thus, the groups are isomorphic to subgroups of the affine group and have a wavelet representation.

In \cite{Cordero} and \cite{Czaja}, two subgroups of the symplectic group $Sp(n+1,\mathbb{R})$, denoted $(CDS)_{n+1}$ and  $(TDS)_{n+1}$
 were shown  to be isomorphic to subgroups of the affine group, and it was shown that their metaplectic representations and wavelet
representations have equivalent subrepresentations. 
In \cite{Namngam}, it was demonstrated that these two groups fall into the class of groups 
 of the form  (\ref{eq:01}), where $D$ is a one-parameter group
\[
	D= D_{p, B}:= \{ \,  d( e^{pt} ,e^{Bt},1) : t \in \mathbb{R} \, \}
\]
for some fixed $p \in \mathbb{R}$ and $B \in M_{n}(\mathbb{R})$, the latter  not similar to a skew-symmetric matrix
in case $p=0$. 
Furthermore, all groups of form $ \mathbb{H}_{pol}^n \rtimes D_{p,B}$ were classified up to isomorphism, 
and were shown to be isomorphic to subgroups of both,
the symplectic group  $Sp(n+1,\mathbb{R})$ as well as the affine group $\textit{Aff}(n+1,\mathbb{R})$.
In addition, it was shown that the metaplectic representation splits into two subrepresentations,
each of which is equivalent to a subrepresentation of the wavelet representation.

In this paper, we continue the study of semi-direct products of type $ \mathbb{H}_{pol}^n \rtimes D_{p,B}$, 
where  $D_{p,B}$ is now a two-parameter group.  After having introduced these groups,
we classify them up to isomorphism by analyzing their Lie algebras.
We then proceed at showing that they are isomorphic to subgroups of $Sp(n+1,\R)$ as well as
$\textit{Aff}(n+1,\R)$, and study their metaplectic and wavelet representations.

%%%%%%%%%%%%%%%%%%%%%%%%%%%%%%%%%%%%%%%%%%%%%%%%%%%%%%%%%%%%%%%%%%%%%%%%%%%%%%%%%%%%
%%%%%%%%%%%%%%%%%%%%%%%%%%%%%%%%%%%%%%%%%%%%%%%%%%%%%%%%%%%%%%%%%%%%%%%%%%%%%%%%%%%%

\section{Extensions of the Heisenberg Group}

%%%%%%%%%%%%%%%%%%%%%%%%%%%%%%%%%%%%%%%%%%%%%%%%%%%%%%%%%%%%%%%%%%%%%%%%%%%%%%%%%%%%
\subsection{The groups $G_{p,B}$}
For given fixed numbers $p_1, p_2 \in \mathbb{R}$ and fixed commuting matrices $B_1, B_2 \in M_n(\mathbb{R})$, let us set
\[		
		p:=(p_1, p_2) \qquad \text{and} \qquad B:=(B_1, B_2).	
\]
We also set 
\[
 D_{p,B} = \left\{ {d(t): =\left[ {\begin{array}{*{20}c}
   e^{pt} & 0 & 0  \\
   0 & e^{Bt} & 0  \\
   0 & 0 & 1  \\
 \end{array} } \right] \ :\  t \in \mathbb{R}^2} \right\}
\]
where $pt$ and $Bt$ denote ''scalar'' products,
\[	\qquad
     pt = p_{1}t_{1}+p_{2}t_{2},\;\; Bt=B_{1}t_{1}+B_{2}t_{2} \quad %(t= \footnotesize{\begin{pmatrix} t_{1}\\ t_{2} \end{pmatrix}  
										\qquad (t=\trp{(t_1,t_2)} \in \mathbb{R}^2 ).
\]
Then $D_{p,B}$ is an abelian (not necessarily closed) subgroup of $ GL_{n+2}(\mathbb{R})$. 
Conjugation by elements of $D_{p,B}$ naturally defines a continuous action $\alpha$ of $\mathbb{R}^2$ on $\mathbb{H}^{n}_{pol}$ by
\begin{equation}\label{eq:5}
%\begin{split}
			\alpha_{t} \bigl (  \, h(x,y,z) \,  \bigr ) 
				:= d(t)\;h(x,y,z)\; d(t)^{-1}
				% &= \begin{bmatrix} e^{pt}  & 0 & 0 \\ 0 & e^{Bt} & 0 \\ 0 & 0 & 1     \end{bmatrix}
				%	\begin{bmatrix}  1 & y^{T}  &  z \\ 0 & I_{n} & x \\ 0 & 0 & 1 \end{bmatrix} 
				%	\begin{bmatrix} e^{-pt}  & 0 & 0 \\ 0 & e^{-Bt} & 0 \\ 0 & 0 & 1     \end{bmatrix} 		\\
				% &=   \begin{bmatrix}  1 &   e^{pt} y^{T} e^{-Bt} &  e^{pt} z \\ 0 & I_{n} & e^{Bt} x \\ 0 & 0 & 1 \end{bmatrix} ,
%\end{split}
%\end{equation} 
%that is,
%\begin{equation}
		%	\alpha_t \bigl ( h(x,y,z) \bigr ) 
		 = 	h \bigl  (e^{Bt} x, e^{pt} [e^{-Bt}]^{T}  y,e^{pt} z \bigr ) .
\end{equation}
We can thus form the semidirect product 
\[	
			G_{p,B} := \mathbb{H}_{pol}^{n} \rtimes_{\alpha} \mathbb{R}^{2}.	
\]
The group operation is given by
\begin{align*}
	\left ( h(x,y,z) ,t \right ) \left ( h(\tilde x,\tilde y,\tilde z) , \tilde t \right ) 
		&= \left ( h  (x,y,z ) \, \alpha_t \left (   h(\tilde x,\tilde y,\tilde z)  \right ), t+\tilde t \right ) \\
		&= \left ( h  (x+ e^{Bt} \tilde x,y+ e^{pt} [e^{-Bt}]^{T}  \tilde y,z +e^{pt} \tilde z +y^{T} e^{Bt} x ) , t+\tilde t \right ) .
\end{align*}
Alternatively, we may represent elements of $G_{p,B}$ as quadruples $g(t,x,y,z)$, in this case the group operation becomes
\begin{equation}\label{eq:4}
g(t, x, y,z) g(\tilde t, \tilde x, \tilde y, \tilde z) 
				=  g(t + \tilde t, x+ e^{Bt} \tilde x,y+ e^{pt} [e^{-Bt}]^{T} \tilde y,z + e^{pt} \tilde z + y^{T}e^{Bt} \tilde{x}).
\end{equation}

%%%%%%%%%%%%%%%%%%%%%%%%%%%%%%%%%%%%%%%%%%%%%%%%%%%%%%%%%%%%%%%%%%%%%%%%%%%%%%%%%%%%
\subsection{The groups $G_{p,B}$ as closed subgroups of $GL_{n+2}(\mathbb{R})$}

 Observe that each group $G_{p,B}$ is isomorphic and homeomorphic to the matrix group
\[		G_{p,B} \simeq \left \{ \, \tilde g(t,x,y,z)=
				\begin{bmatrix} 
						 \left [  \begin{smallmatrix} 
						                     e^{pt} &  y^{T}  e^{Bt}  &   z  \\ 
											 0 & e^{Bt} & x \\[4pt] 
											 0 & 0 & 1 
								\end{smallmatrix} \right ]  & 0\\
								
							0 &  \left [ \begin{smallmatrix} e^{t_1} &  0 \\
											 0 & e^{t_2} 
								\end{smallmatrix}  \right]
				\end{bmatrix}  \, : \, \begin{array}{l} t = \trp{(t_1,t_2)} \in \mathbb{R}^2 \\ x,y \in \mathbb{R}^n\\ z\in \mathbb{R} \end{array}
				  \right \} \subset GL_{n+4}(\mathbb{R}).
\]
This representation of $G_{p,B}$ is  not very useful. Under some mild assumptions on the matrices $B_1$ and $B_2$
it is, however, possible to identify $D_{p,B}$ with $\R^2$, and hence $G_{p,B}$ with a closed subgroup of $GL_{n+2}(\mathbb{R})$.  The main ingredient is the proof of  Lemma 11 in \cite{Bruna} which shows the following:
\begin{lem}\label{rem:2}
Let $M_1$ and $M_2$ be commuting $d \times d$ matrices, and suppose that
\begin{enumerate}
\setlength{\parskip}{0pt}
\setlength{\itemsep}{1pt}
	\item[(M1)] $M_1, M_2$ are linearly independent, and
	\item[(M2)] no nonzero element of $ V_M:=\text{span}(M_1, M_2)$  is similar to a skew-symmetric matrix.
\end{enumerate}
Then the exponential map $exp: M \mapsto e^M$ is an isomorphism and homeomorphism of the additive group $V_M$ onto a closed subgroup of $GL_d(\mathbb{R})$.
\end{lem}

Let us now set
\[		 M_k= \begin{bmatrix}	p_k & 0 & 0 \\ 0 & B_k & 0 \\ 0 & 0 & 0 \end{bmatrix}, \qquad (k=1,2),	\]
and assume from here onwards that the matrices $M_1$ and $M_2$ satisfy the conditions (M1)--(M2).
(This certainly is the case when $p_1=1$, $p_2=0$ and $B_2$ is not similar to a
skew-symmetric matrix. When $p_1=p_2=0$, this is the case if and only if $B_1$ and $B_2$ satisfy (M1)--(M2).)
Applying Lemma \ref{rem:2} we immediately obtain that
the map $t \mapsto d(t)$ is an isomorphism and homeomorphism of $\mathbb{R}^2$ onto $D_{p,B}$ and that $D_{p,B}$ 
is closed in $GL_{n+2}(\mathbb{R})$, and hence by (\ref{eq:5}),
\begin{equation}\label{eq:8}	G_{p,B} \simeq \left \{ \ g(t,x,y,z)=
				\begin{bmatrix} 
						  e^{pt} &  y^{T}e^{Bt}  &   z  \\ 
											 0 & e^{Bt} & x \\ 
											 0 & 0 & 1 
				\end{bmatrix}  \ : \ \begin{array}{l} t  \in \mathbb{R}^2 \\ x,y \in \mathbb{R}^n \\ z\in \mathbb{R} \end{array}
				 \ \right \} \subset GL_{n+2}(\mathbb{R}).
\end{equation}

%%%%%%%%%%%%%%%%%%%%%%%%%%%%%%%
\section{Classification of the Groups $G_{p,B}$}
Observe that under assumptions (M1) and (M2), each group $G_{p,B}$ is a connected, simply connected Lie group.
Thus, two groups $G_{p,B}$ and $G_{\tilde p,\tilde B}$ will be isomorphic if and only their Lie algebras  $\mathfrak{g}_{p,B}$ and  $\mathfrak{g}_{\tilde p, \tilde B}$  are isomorphic.

\subsection{The Lie algebras $\mathfrak{g}_{p,B}$}
Standard computations show that each Lie algebra $\mathfrak{g}_{p,B}$ is of dimension $2n+3$ and isomorphic to the matrix subalgebra of $M_{n+2}(\R)$ of the form
\[			\mathfrak{g}_{p,B} = V_{M} \oplus V_H = \underbrace{V_{M_1} \oplus V_{M_2}}_{V_M}
					 \oplus \underbrace{ \overbrace{V_X \oplus V_Y}^{V_W} \oplus V_Z}_{V_H},
\]
where $V_M$ is a $2-$dimensional abelian subalgebra, $V_H$ is the $2n+1-$dimensional Heisenberg algebra, and
\[				V_{M_1} = \{ s M_1: s \in \R \} , \qquad\qquad V_{M_2} = \{ t M_2 : t \in \R \},
				%\qquad \text{with} 
				%\qquad  M_k = \begin{bmatrix}	p_k & 0 & 0 \\ 0 & B_k & 0 \\ 0 & 0 & 0 \end{bmatrix}
\]
\[	
		V_X = \{  X_x : x  \in \R^n \},   \quad\qquad V_Y = \{  Y_y : y  \in \R^n \}, \quad\qquad V_Z = \{  Z_z : z  \in \R \}  
\]
with
\[
		X_x = \begin{bmatrix}	0 & 0 & 0 \\ 0 & 0 & x \\ 0 & 0 & 0 \end{bmatrix}, \qquad
		Y_y = \begin{bmatrix}	0 & y^{\transp} & 0 \\ 0 & 0 & 0 \\ 0 & 0 & 0 \end{bmatrix}, \qquad
		Z_z = \begin{bmatrix}	0 & 0 & z \\ 0 & 0 & 0 \\ 0 & 0 & 0 \end{bmatrix}.
\]
The only possibly nonzero Lie brackets are determined by
\begin{equation}\label{eq:9}
		[M_k,X_x]=X_{B_k x}, \quad [M_k,Y_y]=Y_{(p_kI-B_k^{\transp}) y}, \quad [M_k,Z_z]=Z_{p_k z}, \quad [Y_y,X_x] = Z_{y^{\transp} x},	
\end{equation}
$k=1,2$. For the purpose of classifying this type of Lie algebras we do not require condition (M2).
Observe that these Lie algebras are solvable, and that $V_H$ is an ideal of the nilradical.

It will be convenient to   denote elements $X_x + Y_y$ of $V_W$ 
by $W_w$, where $w= \left [ \begin{smallmatrix} x \\ y \end{smallmatrix} \right ] $. In this notation,  
some of the Lie brackets in (\ref{eq:9}) become
\begin{equation*}
	[W_w,W_{\tilde w}] = [ X_x + Y_y, X_{\tilde x}+ Y_{\tilde y} ]
			=  Z_{y^{\transp} \tilde x + \tilde y^\transp x}   = Z_{ \llbracket w,  \tilde w \rrbracket },
\end{equation*}
and also
\begin{equation}\label{eq:9a}
	[M_k,W_w] =[M_k,X_x] + [M_k,Y_y] = X_{B_k x} + Y_{(p_k I_n-B_k^{\transp}) y} = W_{C_kw}
\end{equation}
with
\begin{equation}\label{eq:10}
			C_k= \begin{bmatrix}  B_k & 0 \\ 0 & p_k I_n -B_k^{\transp} \end{bmatrix} \in M_{2n}(\R).
\end{equation}

\bigskip
The following lemma is probably well known.

\begin{lem}\label{lem1}
Let a triple $(\lambda,u,S)$ be given, where $\lambda >0$,  $u \in \R^{2n}$, and $S \in GL_{2n}(\R)$ 
satisfies $S^{\transp} \mathcal{J} S= \pm \mathcal{J}$.  Then
\begin{equation}\label{eq:10c}
		\qquad\qquad\qquad
		\Phi( W_w ) = W_{\lambda Sw} + Z_{u^{\transp} w} \quad \text{and} \quad \Phi( Z_z ) = Z_{ \pm \lambda^2 z}	
		\qquad\qquad  (W_w \in V_W, \, Z_z \in V_Z) 
\end{equation}
defines an automorphism of the Heisenberg algebra $V_H$. 
Conversely, every automorphism of $V_H$ is of this form.
\end{lem}
\begin{proof}
It is clear that the linear map $\Phi$ defined by (\ref{eq:10c}) constitutes a linear
automorphism of $V_H$. On the other hand, by assumption on $S$ we have for all $w, \tilde w \in \R^{2n}$,
\begin{equation}\label{eq:10d}
\begin{split}	
		\left [\Phi( W_{w}), \Phi(W_{\tilde w}) \right ]
		& =  \left [ W_{\lambda Sw} +  Z_{u^{\transp} w }, W_{\lambda S \tilde w} +  Z_{u^{\transp} \tilde w} \right ] 
			= Z_{\llbracket \lambda Sw, \lambda S \tilde w \rrbracket }   \\
		&	= Z_{ \pm \lambda^2 \llbracket w,  \tilde w \rrbracket } 
			= \Phi \left ( Z_{\llbracket w,  \tilde w \rrbracket } \right )
			= \Phi \left ( \left [ W_{w}, W_{\tilde w} \right ] \right ),
\end{split}
\end{equation}
and it follows that $\Phi$   preserves the Lie brackets. 

Conversely, let $\Phi$ be a Lie algebra automorphism of $V_H$. In light of the decomposition $V_H = V_W \oplus V_Z$
and since $\Phi$ leaves the center $V_Z$ invariant, $\Phi$ has a matrix representation
\[			
			\Phi  \leftrightarrow \begin{bmatrix}  a_{11} & 0 \\ a_{21} & a_{22 } \end{bmatrix}
\]
where $a_{11} \in GL_{2n}(\R)$ and $a_{22} \neq 0$. Computing as in (\ref{eq:10d}) we have for all $w,\tilde w \in \R^{2n}$,
\begin{equation*}%\label{eq:10e}
\begin{split}	
		 Z_{ a_{22} \llbracket  w,  \tilde w \rrbracket }  
		 &= \Phi \left ( Z_{\llbracket w,  \tilde w \rrbracket } \right )
		 = \Phi \left ( \, \left [ W_{w}, W_{\tilde w} \right ] \, \right )
		 = \left [\Phi( W_{w}), \Phi(W_{\tilde w}) \right ] \\
		& =  \left [ W_{ a_{11}w } +  Z_{ a_{21}  w }, W_{a_{11}\tilde w} +  Z_{a_{21} \tilde w} \right ] 
			= Z_{\llbracket a_{11}w, a_{11} \tilde w \rrbracket }   \\
\end{split}
\end{equation*}	
Set $\lambda = \sqrt{ |a_{22}| }$, $S = \frac{1}{\lambda} a_{11}$ and  $u =a_{21}^{\transp}$.
Then $ \llbracket Sw, S \tilde w \rrbracket = \text{sgn}(a_{22})  \llbracket w,  \tilde w \rrbracket $, that is
$ S^{\transp} \mathcal J S = \text{sgn}(a_{22}) \mathcal{J} $, and the assertion follows.
\end{proof}

\subsection{Classification of the the Lie algebras $\mathfrak{g}_{p,B}$}

Let us first introduce some normalization to the class of Lie algebras $\mathfrak{g}_{p,B}$.
Given two algebras $\mathfrak{g}_{p,B}$ and $\mathfrak{g}_{\tilde p, \tilde B}$, their Heisenberg parts
are identical, so we will  use the same symbol $V_H$ to denote the two, and  the
remaining component spaces will be denoted by $V_M$ and $V_{\tilde M}$, respectively.

\begin{thm}\label{thm:1}
If any of the following properties hold, then two Lie algebras
 $\mathfrak{g}_{p,B}$ and $\mathfrak{g}_{\tilde p, \tilde B}$  are isomorphic:
\begin{enumerate}
\item   $\tilde p = p$ and there exists  $S \in Sp(n,\R)$ so that 
				\[		\tilde C_k= S C_k S^{-1}  \qquad \qquad (k=1, 2) ,	\]
				with $C_k$ and $\tilde C_k$ given as in (\ref{eq:10}).
\item   $\tilde p = p$ and there exists   $V \in GL_{n}(\R)$ so that 
				\[		\tilde B_k= V B_k V^{-1}  \qquad \qquad (k=1, 2) .	\]
\item   Each $\tilde M_i$ is a linear combination of $M_1$ and $M_2$,
		\[	\tilde M_i = a_{i1} M_1 + a_{i2} M_2	\qquad \qquad (i=1, 2) \]
		with $\det(A) \neq 0$ where $A= \left[ \begin{smallmatrix} a_{11} & a_{22} \\ a_{21} & a_{22} \end{smallmatrix} \right ] $.
\item 	There exists $\alpha \neq 0$ so that $\tilde M_k = \alpha M_k$ for $k=1, 2$.
\end{enumerate}
\end{thm}
\begin{proof}
\begin{enumerate}
\item Define a linear isomorphism 
		$\Phi :  \mathfrak{g}_{p, B} \to  \mathfrak{g}_{\tilde p, \tilde B}$ by
\[	
				\Phi(M_k) = \tilde M_k, \qquad \Phi(W_w) = W_{Sw}, \qquad   \Phi(Z_z) = Z_z .
\]
In light of Lemma \ref{lem1} one only needs to verify that Lie brackets involving
the matrices $M_k$ are preserved. This is indeed the case, as  by (\ref{eq:9a}), for all $k=1,2$,
\[	\qquad\quad
	\left [ \Phi(M_k), \Phi( W_w) \right ] 
			= \ [ \tilde M_k,  W_{Sw}   ] 
			=   W_{\tilde C_k Sw} = W_{ S C_k w} = \Phi \left (W_{C_k w} \right ) = \Phi \left ( [M_k,W_w] \right )
\]
and
\[
	\left [ \Phi(M_k), \Phi( Z_z) \right ] = \ [ \tilde M_k, Z_z   ] 
			=  Z_{\tilde p_kz}  = Z_{p_kz} = \Phi \left ( [M_k,Z_z] \right ).
\]
\item Simply apply the above to 
		\[  S = \begin{bmatrix}   V & 0 \\ 0  & \left (V^{-1}\right )^{\transp} \end{bmatrix} .	\]
\item  This is merely a change of basis of the subalgebra $V_M$, and hence both Lie algebras coincide.
\item  This is a particular change of basis, choosing $a_{ik} = \alpha \delta_{i,k}$.
\end{enumerate}
%Thus,  the assertions hold.
\end{proof}

Replacing the matrices $M_1, M_2$ (and consequently $B_1, B_2$) with appropriate linear combinations, by Theorem \ref{thm:1}, we may  from here on
assume that $p_1 \in \{0,1\}$ and $p_2=0$. After this normalization, we aim to give a converse of Theorem \ref{thm:1}. 

\begin{rem}\label{rem:1} If two normalized Lie algebras $\mathfrak{g}_{p,B}$, and $ \mathfrak{g}_{\tilde p, \tilde B}$ are isomorphic, then $p_1 = \tilde p_1$ (i.e. $p=\tilde p$). In fact, if $ \Phi : \mathfrak{g}_{p,B} \mapsto \mathfrak{g}_{\tilde p, \tilde B}$  is a Lie algebra isomorphism, then $\Phi$ maps center onto center. Since $\mathfrak{g}_{ p,  B}$ has trivial center when $p_1=1$, and center $V_Z$ when $p_1=0$, it immediately follows that $p_1=\tilde p_1$.
\end{rem}

This remark shows that the normalized Lie algebras $\mathfrak{g}_{p,B}$ need only be classified with respect to the various 
choices of $B$.

\begin{thm}\label{thm:2}

Let  $ \Phi : \mathfrak{g}_{p,B} \mapsto \mathfrak{g}_{p, \tilde B}$ be an isomorphism of normalized Lie algebras mapping $V_H$ onto $V_H$. 
Then there exists $S \in Sp(n,\R)$ so that, after replacing the matrices  $\tilde M_1, \tilde M_2$
			with suitable linear combinations,
				\begin{equation*}
					\tilde C_k 	= S C_k S^{-1}, \qquad k=1, 2,
				\end{equation*}
				with $C_k$ and $\tilde C_k$ given as in (\ref{eq:10}).
\end{thm}

\begin{proof}
Suppose that $\Phi: V_H \mapsto V_H$. Then in light of Lemma \ref{lem1}, $\Phi$ has the matrix representation
\begin{equation}\label{eq:14}
		\Phi \leftrightarrow \begin{bmatrix} E_{11} & 0 & 0 \\ E_{21} & E_{22} & 0 \\ E_{31} & E_{32} & E_{33} \end{bmatrix},
\end{equation}
corresponding to the decomposition $ \mathfrak{g}_{ p,  B}=V_M \oplus V_W \oplus V_Z$.
Note that composing $\Phi$  with the automorphism $\Psi$ of $ \mathfrak{g}_{ p, \tilde B}$ given by the matrix
\[
	\Psi \leftrightarrow \begin{bmatrix} I_2 & 0 & 0 \\ 0 & \lambda \mathcal{J} & 0 \\ 0 & 0 & -\lambda^2 \end{bmatrix},
		 \qquad \text{resp.} \qquad
		\Psi \leftrightarrow \begin{bmatrix} I_2 & 0 & 0 \\ 0 & \lambda I_{2n} & 0 \\ 0 & 0 & \lambda^2 \end{bmatrix},
\]
depending on the sign of $E_{33}$ and with $\lambda = |E_{33}|^{-1/2}$, we may assume that $E_{33}=1$.

After a suitable change of basis in $V_{\tilde M}$, which  affects the first column of   matrix (\ref{eq:14}) only,
we may assume that $E_{11}=I_2$. 
It is important to observe that this change of basis can be achieved without changing the values of $p_k$.
This is clear when $p_1 =0$.
On the other hand, suppose that $p_1 =1$. Now if 
$E_{11}= \left [ \begin{smallmatrix} e_{11} &  e_{12} \\ e_{21} & e_{22} \end{smallmatrix} \right ]$, then 
$\Phi(M_k) =e_{1k} \tilde M_1 + e_{2k} \tilde M_2 + H_k$ for some $H_k \in V_H$ and it follows that for all $z \in \R$,
\begin{equation}\label{eq:15}
	\left [ \Phi(M_k), \Phi(Z_z) \right ] = e_{1k}  \left [  \tilde M_1 ,Z_{E_{33}z} \right ]  + e_{2k}  \left [  \tilde M_2 ,Z_{E_{33}z} \right ]
		+ \left[  H_k, Z_{E_{33}z} \right ] 
		= e_{1k} Z_{E_{33}z}
\end{equation}
while also
\begin{equation}\label{eq:15a}
	\left [ \Phi(M_k), \Phi(Z_z) \right ] = \Phi\left ( \left [ M_k, Z_z \right ] \right )
		= \Phi( \delta_{1,k} Z_z ) = \delta_{1,k} Z_{E_{33}z}.
\end{equation}
Comparing these two equations we obtain that 
\[
		e_{11}=\delta_{1,1}=1, \qquad  e_{12}=\delta_{1,2}=0.
\] 
Now replacing  $\tilde M_1$ with $\tilde M_1 + e_{21} \tilde M_2$ and then scaling $\tilde M_2$
we arrive at $E_{11}=I_2$, without changing the values of $ p_k$.
The isomorphism $\Phi$  now has the form
\[
		\Phi \leftrightarrow \begin{bmatrix} I_2 & 0 & 0 \\ E_{21} & E_{22} & 0 \\ E_{31} & E_{32} & 1 \end{bmatrix}
\]
with $E_{22} \in GL_{2n}(\R)$.

It is easy to verify that a linear isomorphism determined by such a matrix preserves the
Lie brackets if and only if
\begin{gather}
					\llbracket E_{22} w, E_{22} \tilde w \rrbracket  =  \llbracket w,\tilde w \rrbracket 	\label{eq:22} \\
					\tilde C_k 	= E_{22} C_k E_{22}^{-1} 				\label{eq:23}\\
				  	\tilde C_1 E^{(2)}_{21} = \tilde C_2 E^{(1)}_{21}			\notag \\
				  	p_1 E^{(2)}_{31} + \llbracket E^{(1)}_{21}, E^{(2)}_{21} \rrbracket = p_2 E^{(1)}_{31} 	\notag	\\
				  	\llbracket E^{(k)}_{21}, w \rrbracket = E_{32}E_{22}^{-1} w	\notag
\end{gather}
for $k=1, 2$ and $w, \tilde w \in \R^{2n}$,  with $E^{(k)}_{21}$ and  $E^{(k)}_{31}$ denoting the $k$-th
columns of the matrices $E_{21}$ and $E_{31}$, respectively.
These identities remain valid if we modify $\Phi$ so that $E_{21}=E_{31}=E_{32}=0$, that is
\begin{equation}\label{eq:23:b}
		\Phi \leftrightarrow \begin{bmatrix} I_m & 0 & 0 \\ 0 & E_{22} & 0 \\ 0 & 0 & 1 \end{bmatrix}.
\end{equation}
Choosing $S=E_{22}$, the identities (\ref{eq:22}) and (\ref{eq:23}) now yield the assertion.
\end{proof}

We now investigate properties of Lie algebra isomorphisms between two normalized Lie algebras. Any  isomorphism $ \Phi: \mathfrak{g}_{p,B} \mapsto \mathfrak{g}_{p,\tilde B}$ between two Lie algebras
of this type can be represented in matrix form as
\begin{equation}\label{eq:29c}
		 \Phi \leftrightarrow  \begin{bmatrix}
							a_{11} & a_{12} & a_{13} & a_{14}\\
							a_{21} & a_{22} & a_{23} & a_{24}  \\
							a_{31} & a_{32} & a_{33} & a_{34}\\
							a_{41} & a_{42} & a_{43} & a_{44}
							\end{bmatrix},
\end{equation}
by using the decomposition $\mathfrak{g}_{p,B} =V_{M_1} \oplus V_{M_2} \oplus V_W \oplus V_Z$. 
Our goal is to show that $a_{13}=a_{23}=a_{14}=a_{24}=0$, which guarantees that $\Phi$ maps $V_H$ onto $V_H$.
We begin with the following observation.

\begin{lem}\label{lem2}
Let $ \Phi: \mathfrak{g}_{p,B} \mapsto \mathfrak{g}_{p,\tilde B} $ 
be a Lie algebra isomorphism which  has the matrix representation
\begin{equation}\label{eq:30}
		 \Phi \leftrightarrow  \begin{bmatrix}
							a_{11} & 0 & 0  & 0\\
							0 & a_{22} & a_{23} & 0  \\
							a_{31} & a_{32} & a_{33} & 0\\
							a_{41} & a_{42} & a_{43} & a_{44}
							\end{bmatrix}.
\end{equation}
Then $a_{23}=0$.
\end{lem}

\begin{proof}
Since $\Phi$ maps the ideal $V_Z$ onto $V_Z$, it factors to a Lie algebra isomorphism 
$\hat \Phi : \mathfrak{h} = \mathfrak{g}_{p,B}/V_Z \simeq V_{M_1} \oplus V_{M_2} \oplus V_W
		\mapsto \tilde{\mathfrak{h}}= \mathfrak{g}_{p, \tilde B}/V_Z \simeq V_{\tilde M_1} \oplus V_{\tilde M_2} \oplus V_W$ 
		whose matrix representation is 
\begin{equation*}
		 \hat \Phi \leftrightarrow  \begin{bmatrix}
							a_{11} & 0 & 0 \\
							0 & a_{22} & a_{23}  \\
							a_{31} & a_{32} & a_{33}
							\end{bmatrix}.
\end{equation*}
Let us set $\mathfrak{k}=V_{M_2} \oplus V_W$. Then $\hat  \Phi$ maps $\mathfrak{k}$ onto $\tilde{\mathfrak{k}}$, 
and $V_W$ is an abelian ideal of codimension one in  $\mathfrak{k}$, respectively  $\tilde{\mathfrak{k}}$.

We claim that $V_W$ is the unique such ideal. 
For suppose that $J$ is another abelian ideal of codimension one in  $\mathfrak{k}$. Let $U_1,\dots,U_{2n}$ be a basis of $J$. Then each
$U_i$ is of the form 
\[
				U_i = \alpha_i M_2 + W_{w_i}, \quad \alpha_i \in \R, \ i=1,\dots,2n.
\]
If $\alpha_i=0$ for all $i$, the claim is proved, otherwise we may assume without loss of generality, that $\alpha_1=1$
and $\alpha_i=0$ for all $i \ge 2$. Now since $B_2 \neq 0$, there exist $x_o, y_o \in \R^n$ so that
\begin{align*}
		[ U_1, X_{x_o} ] &= [ M_2 , X_{x_o}] = X_{B_2 x_o} \neq 0\\
		[ U_1, Y_{y_o} ] &= [ M_2 , Y_{y_o}] = Y_{-B_2^{\transp} y_o} \neq 0.
\end{align*}
Since $J$ is abelian, it follows that $X_{x_o},Y_{y_o} \notin J$, contradicting the assumption that $\text{codim}(J)=1$. This proves the claim.

From the claim it follows immediately that $\hat \Phi$ maps $V_W$ onto  $V_W$, and hence that $a_{23}=0$.
\end{proof}

\begin{thm}\label{thm:5}
Let $ \Phi: \mathfrak{g}_{p,B} \mapsto \mathfrak{g}_{p,\tilde B}$ be a Lie algebra isomorphism of normalized Lie algebras. Then $\Phi$ maps $V_{H}$ onto $V_{H}$.
\end{thm}

\begin{proof}
We consider five distinct possibilities:
 $p_1=1$ and $B_2$ is nilpotent, $p_1=1$ and $B_2$ is not nilpotent, $p_1=0$ and
 none of $B_1$ and $B_2$ is nilpotent,
 $p_1=0$ and exactly one of $B_1$ and $B_2$ is nilpotent, 
and  $p_1=0$ and both, $B_1$ and $B_2$ are nilpotent.

As will be seen below, in each of the five cases, $\mathfrak{g}_{p,B}$ will have a different
algebraic structure. Thus, two Lie algebras which are isomorphic via some isomorphism $\Phi$
must belong to the same of the five cases.

\setitemize[1]{leftmargin=16pt}
\setitemize[2]{leftmargin=12pt}
\setitemize[3]{leftmargin=12pt}

\begin{itemize}
\item {\it Case 1: $p_1=1$ and $B_2$ is not nilpotent}  

	Here,  $\mathfrak{g}_{p,B}$ has nilradical $V_H$ of dimension $2n+1$. Since $\Phi$ maps nilradical to nilradical, it follows that 	 		 
	$\mathfrak{g}_{p,\tilde B}$  has nilradical of dimension $2n+1$ as well, which thus must coincide with $V_H$. 
	That is, $\Phi$ maps $V_H$ 	onto $V_H$.

\item {\it Case 2: $p_1=1$ and $B_2$ is nilpotent}

Here,  $\mathfrak{g}_{p,B}$ has nilradical $V_{M_2} \oplus V_H$ of dimension $2n+2$. Since $p_1=1$, and the nilradical of $\mathfrak{g}_{p,\tilde B}$ has dimension $2n+2$ as well, it follows that $\tilde B_2$ is nilpotent and the nilradical of  $\mathfrak{g}_{p,\tilde B}$ is  $V_{\tilde M_2} \oplus V_H$. In addition, as $\Phi$ maps the center $V_Z$ of the nilradical onto the center $V_Z$ of the nilradical, it follows that $\Phi$ has matrix form
\begin{equation}\label{eq:38}  
\Phi \leftrightarrow
		\begin{bmatrix}  a_{11} &  0  & 0 & 0 \\
						 a_{21} & a_{22} & a_{23} & 0\\
						 a_{31} & a_{32} & a_{33} & 0\\
						 a_{41} & a_{42} & a_{43} & a_{44}
		\end{bmatrix}.
\end{equation}
	Replacing  $\tilde M_1$ with a suitable $\tilde M_1 + \beta \tilde M_2$ we may assume that $a_{21}=0$. Applying Lemma \ref{lem2} it follows that $a_{23}=0$, that is, $\Phi$ maps $V_H$ onto $V_H$.

\item {\it Case 3: $p_1=0$ and none of $B_1,B_2$ is nilpotent} 

Simply apply the argument of case~1.

\item {\it Case 4: $p_1=0$ and one of $B_1,B_2$ is nilpotent}

Without loss of generality, we may assume that $B_2$ is nilpotent, but $B_1$ is not. Then $\mathfrak{g}_{p,B}$ has nilradical $V_{M_2} \oplus V_H$ of dimension $2n+2$. Since $p_1=0$ and the nilradical of $\mathfrak{g}_{p,\tilde B}$ has dimensions $2n+2$, the latter algebra  must again belong to case 4, so that replacing $\tilde B_1$ and $\tilde B_2$ by suitable linear combinations,$\mathfrak{g}_{p,\tilde B}$ will have nilradical $V_{\tilde M_2} \oplus V_H$.  The remainder of the argument follows that of case 2.

\item {\it Case 5: $p_1=0$ and both, $B_1$ and $B_2$ are nilpotent}

Here, $\mathfrak{g}_{p,B}$ is itself nilpotent with center $V_Z$. Hence  $\mathfrak{g}_{p,\tilde B}$ is also nilpotent and again 
belongs to case 5. Since $\Phi$ maps center to center, 
it has the form (\ref{eq:29c}) with $a_{14}=a_{24}=a_{34}=0$. 
We begin by considering the induced isomorphism $\hat \Phi : \mathfrak{h} \mapsto \tilde{\mathfrak{h}}$,
\[
	\hat \Phi \leftrightarrow 	\begin{bmatrix}  
								a_{11} &  a_{12} & a_{13} \\
								a_{21} & a_{22} & a_{23} \\
								a_{31} & a_{32} & a_{33} 
				\end{bmatrix}.
\] 
Since $  V_W$ is an ideal of codimension two in $\mathfrak{h}$,
then $\tilde I:=\hat \Phi(V_W)$ will be an abelian ideal of codimension two in $\tilde{\mathfrak{h}}$, that is, of dimension $2n$. 

We claim that $\tilde I=V_W$. Suppose to the contrary that $\tilde I \not =V_W$. 
Denoting by $P_o$ the projection of $\mathfrak{h}$ onto $V_{\tilde M}=V_{\tilde M_1} \oplus V_{\tilde M_2}$ and setting $V_o=P_o(\tilde I)$,
we then obtain that $\dim(V_o) \in \{1,2 \}$.

\begin{itemize}

\item [$\circ$] {\it Subcase 5a: $\dim(V_o)=1$.} 
		Then elements of $\tilde I$ are of the form
\[	
		A = \alpha \tilde M_o + H ,\qquad \alpha \in \R, \ H \in V_W	
\]
for some fixed nonzero $\tilde M_o=\text{diag}(0,\tilde B_o,0) \in V_{\tilde M}$. 
Fix one such $A$ with $\alpha =1$. Then there exist $x_o,y_o \in \R^n$ so that
\begin{align*}	
	[ A, X_{x_o} ] &=[ \tilde M_o, X_{x_o} ]= X_{\tilde B_o x_o} \neq 0	\qquad \text{and} \\
	[ A, Y_{y_o} ] &= [ \tilde M_o, Y_{y_o} ] = Y_{-\tilde B_o^{\transp} y_o} \neq 0.	
\end{align*}
Since $\tilde I$ has codimension two, it follows that $\mathfrak{h}=  \tilde I \oplus <X_{x_o},Y_{y_o}>$ where  $<X_{x_o},Y_{y_o}>$ denotes $\text{span}(X_{x_o},Y_{y_o})$. In fact, suppose $\alpha X_{x_o} +\beta Y_{y_o} \in \tilde I$ for some scalars $\alpha,\beta$ . Then
\[ 
		0 = \left [ A,  \alpha X_{x_o} + \beta Y_{y_o} \right ]
 			 =    \alpha [ A,  X_{x_o}  ] + \beta   [ A,  Y_{y_o}   ]
  			=   \alpha  X_{\tilde B_o x_o} +  \beta Y_{-\tilde B_o^{\transp} y_o} ,
\]
which implies that $\alpha=\beta=0$. Now as $<X_{x_o},Y_{y_o}> \subseteq V_W$ we have
\[		
		V_o = P_o( \tilde I) = P_o( \tilde I \oplus <X_{x_o},Y_{y_o}> ) = P_o( \mathfrak{h}) = V_{\tilde M}	
\]
contradicting the fact that $V_o$ has dimension one.

\item [$\circ$] {\it Subcase 5b: $\dim(V_o)=2$.} 

Then  $V_o= V_{\tilde M}$. Note that by nilpotency of $\tilde B_1$ and $\tilde B_2$,
		all linear combinations $\alpha \tilde B_1 + \beta \tilde B_2$ have nontrivial null space.
							   
\begin{itemize}
\item [$\diamond$] {\it Subcase 5b-1: there exists $\tilde B_o=\alpha \tilde B_1 + \beta \tilde B_2$ whose null space has dimension $\leq n-2$.}
				
Set $\tilde M_o=\alpha \tilde M_1 + \beta \tilde M_2$ and pick any $A \in \tilde I$ with $P_o(A)=\tilde M_o$.
By choice of $\tilde B_o$, there exist two elements $ x_1,x_2 \in \R^n$ 
with $[ \tilde M_o, X_{x_1} ] = X_{\tilde B_o x_1}$ and  $[ \tilde M_o, X_{x_2} ] = X_{\tilde B_o x_2}$ linearly independent. Also, pick  $y_1 \in \R^n$ with $[\tilde M_o, Y_{y_1}]= Y_{-\tilde B_o^{\transp} y_1} \neq 0$.
We observe that  $\tilde I + <X_{x_1},X_{x_2},Y_{y_1}>$ is a $2n+3$ dimensional subspace of $\mathfrak{h}$. 
In fact, suppose $   \alpha X_{x_1} +\beta X_{x_2} + \gamma Y_{y_1} \in \tilde I$ for some scalars $\alpha,\beta, \gamma $. Then
\begin{align*}
0& = \left [ \tilde M_o,  \alpha X_{x_1} +\beta X_{x_2} + \gamma Y_{y_1} \right ]\\
&=     \alpha [ \tilde M_o,  X_{x_1}  ] + \beta   [ \tilde M_o, X_{x_2} ] + \gamma [ \tilde M_o,  Y_{y_1}   ] \\
&= \alpha X_{\tilde B_o x_1} +  \beta X_{\tilde B_o x_2} + \gamma Y_{-\tilde B_o^{\transp} y_1} 
\end{align*}
from which it follows that $\alpha=\beta=\gamma=0$. This, however, contradicts the fact that $\mathfrak{h}$ has dimension $2n+2$.
				
\item [$\diamond$] {\it Subcase 5b-2: the null spaces of all nonzero $ \alpha \tilde B_1 + \beta \tilde B_2$ have dimensions $n-1$.}

Pick elements $A_1 = \tilde M_1 + H_1$ and $A_2 = \tilde M_2 + H_2$ ($H_1,H_2 \in V_W$) of $\tilde I$. Since 	
\begin{equation}\label{eq:44}
  \text{ad}( A_i)(X_x)=X_{\tilde B_i x} \qquad \text{and} \qquad \text{ad}(A_i)(Y_y)=Y_{-\tilde B_i^{\transp} y}, \qquad i=1,2,
\end{equation}
it follows 	 that $\ker(\text{ad}(A_1))$ and $\ker(\text{ad}(A_2))$ both have codimensions of at least 2 in $\mathfrak{h}$.
In addition, since $\tilde I$ is abelian, then $\tilde I \subseteq \ker(\text{ad}(A_1)) \cap  \ker(\text{ad}(A_2))$. Comparing dimensions, if follows that 	$ \tilde I = \ker(\text{ad}(A_1)) =  \ker(\text{ad}(A_2)) $.
Now (\ref{eq:44}) shows that $\ker(\text{ad}(A_i)_{|V_W})$ splits into subspaces $V_{X_o}$ and $V_{Y_o}$
	of $V_X$, respectively  $V_Y$, of codimensions one. Hence we can decompose $V_X$ and $V_Y$ as direct sums
	\begin{equation}\label{eq:45}
			V_X = V_{X_o } \oplus < X_{x_o} > , \qquad V_Y = V_{Y_o}  \oplus <Y_{y_o}>
	\end{equation}
	of subspaces. Here we have chosen the vectors $x_o$ and $y_o$ so that $x_o \perp X_o$ and $y_o \perp Y_o$ in $\R^n$
	with respect to the usual inner product. 
	Now since 	$X_o = \ker(\tilde B_i)$ and also $Y_o = \ker(\tilde B_i^{\transp})=\text{range}(\tilde B_i)^{\perp}$ ($i=1,2$), 
	it follows that, after  expressing the 	common domain space as $X_o \oplus <x_o>$ and the common range space 
	as $Y_o \oplus <y_o>$, the matrices $\tilde B_i$ 
	take the form
	\[		
				\tilde B_i = \begin{bmatrix} 0 & 0 \\ 0 & b_i \end{bmatrix}
	\]
	for scalars $b_1$ and $b_2$, contradicting the linear independence of the two matrices.
\end{itemize}
 \end{itemize}
Thus, the claim is proved. 
It follows immediately that $a_{13}=a_{23}=0$. 
Since $\Phi$ maps center $V_Z$ onto center $V_Z$, then also $a_{14}=a_{24}=a_{34}=0$.
 That is, $\Phi$ maps $V_H$ onto $V_H$.
 \end{itemize}
 This completes the proof.
\end{proof}

Combining  Theorems \ref{thm:1}, \ref{thm:2}, \ref{thm:5} and Remark \ref{rem:1}, we arrive at:

\begin{cor}
Two normalized Lie algebra $\mathfrak{g}_{p,B}$ and $\mathfrak{g}_{\tilde{p},\tilde{B}}$ are isomorphic if and only if 
\begin{enumerate}
	\item $p=\tilde p$, and
	\item  there exists $S \in Sp(n,\R)$ so that, after replacing the matrices  $\tilde M_1, \tilde M_2$
			with a suitable basis of  $V_{\tilde M}$,
				\begin{equation*}%\label{eq:20}		
					\tilde C_k 	= S C_k S^{-1}, \qquad k=1, 2,
				\end{equation*}
				with $C_k$ and $\tilde C_k$ given as in (\ref{eq:10}).
\end{enumerate}
\end{cor}

Table \ref{tab:1} lists the equivalence classes of all Lie algebras $\mathfrak{g}_{p,B}$
in the lowest dimensions, namely for $n=1,2$. We note that the non-nilpotent cases can
also be obtained from the list in \cite{Rubin}.

\begin{center}
\begin{table}[ht]
\caption{Equivalence classes of $\mathfrak{g}_{p,B}$ for $n=1,2$.}
\label{tab:1}
\[
\begin{array}{l|ccll}
\hline
	& B_1 & B_2 & \text{Range of parameters} & \text{Remarks} \\
\hline
\hline
n=1	&&&\\
	\quad p=0 &   \text{---} &  \text{---} & & \text{none exists}\\[8pt]
	\quad p=1 & \begin{bmatrix}	   \frac{1}{2}  	\end{bmatrix}	
	   					&    \begin{bmatrix} 1 		\end{bmatrix}		 	
				 		& 	\\[8pt]
\hline
n=2	&&&\\
	\quad p=0 &  \begin{bmatrix}	  1 & 0 \\ 0 & 0  		\end{bmatrix}	
	   					& \begin{bmatrix} 	  0 & 0 \\  0 & 1	\end{bmatrix}		 	
				 		&  \\[16pt]
			&  \begin{bmatrix} 	  1 & 0 \\   0 & 1		\end{bmatrix}	
	   					& \begin{bmatrix} 	  0 & 1 \\  0 & 0	\end{bmatrix}		 	
				 		&  & \text{$B_2$ is nilpotent} \\[16pt]
			&  \begin{bmatrix} 	  1 & 0 \\   0 & 1		\end{bmatrix}	
	   					& \begin{bmatrix} 	  0 & 1 \\  -1 & 0	\end{bmatrix}		 	
				 		&   \\[20pt]
	\quad p=1 & \begin{bmatrix} 	  \frac{1}{2} & 0 \\  0 & b		\end{bmatrix}	
	   					& \begin{bmatrix} 	  1 & 0 \\   0 & d		\end{bmatrix}		 	
				 		& \begin{array}{l} b >  \frac{1}{2}, \ 0 \leq |d| \leq 1 \\[3pt]
				 						b =  \frac{1}{2}, \ 0 \leq d \leq 1 
				 			\end{array} &    \\[16pt]
%				 & \begin{bmatrix} 	  \frac{1}{2} & 0 \\   0 & \frac{1}{2}		\end{bmatrix}	
%	   					& \begin{bmatrix}   1 & 0 \\   0 & d		\end{bmatrix}		 	
%				 		& 0 \leq  d \leq 1	\\[16pt]	
				 & \begin{bmatrix} 	  \frac{1}{2} & 1 \\   0  & \frac{1}{2}		\end{bmatrix}	
	   					& \begin{bmatrix}  1 & d \\  0 & 1		\end{bmatrix}		 	
				 		& d \ge 0&
				 			 	\\[16pt]			
				 & \begin{bmatrix}  \frac{1}{2} & 0 \\   0  & \frac{1}{2}		\end{bmatrix}	
	   					& \begin{bmatrix}  1 & 1 \\  0 & 1		\end{bmatrix}		 	
				 		&	\\[16pt]	 		
				 & \begin{bmatrix}   a & 0 \\   0  & a		\end{bmatrix}	
	   					& \begin{bmatrix}  0 & 1 \\   0 & 0		\end{bmatrix}		 	
				 		&a \ge \frac{1}{2}   & \text{$B_2$ is nilpotent}    \\[16pt]
				 & \begin{bmatrix}  a & 0 \\   0  & a		\end{bmatrix}	
	   					& \begin{bmatrix}  c & 1 \\  -1 & c		\end{bmatrix}		 	
				 		&a \ge \frac{1}{2} , \  c \ge 0  &  \\[16pt]
				 & \begin{bmatrix} \frac{1}{2} & b \\  -b  & \frac{1}{2}		\end{bmatrix}	
	   					& \begin{bmatrix}  1 & 0 \\   0 & 1		\end{bmatrix}		 	
				 		&b > 0  & 			  \text{case $b=0$ covered above}  
 				 \\[12pt]
\hline
\hline
\end{array}
\]
\end{table}
\end{center}

%%%%%%%%%%%%%%%%%%%%%%%%%%%%%%%%%%%%%%%%%%%%%%%%%%%%%%%%%%%%%%%%%%%%%%%%%%%%%%
%%%%%%%%%%%%%%%%%%%%%%%%%%%%%%%%%%%%%%%%%%%%%%%%%%%%%%%%%%%%%%%%%%%%%%%%%%%%%%
\section{Representations of the Groups $G_{p,B}$}

In this section we show that the groups $G_{p,B}$ can be represented as subgroups of both, the
symplectic group $Sp(n+1,\R)$, as well as the affine group $\text{\it Aff}(n+1)$.
Thus, they posses both, a metaplectic and a wavelet representation.
We also show that the metaplectic representation is equivalent to a sum of two copies of a subrepresentation of the wavelet representation.

%%%%%%%%%%%%%%%%%%%%%%%%%%%%%%%%%%%%%%%%%%%%%%%%%%%%%%%%%%%%%%%%%%%%%%%%%%%%%%
\subsection{Preliminaries}

\textbf{Notation.} Throughout, symbols $ x ,y$  will denote vectors in  Euclidean space  $\R^n$ written as column vectors, 
while symbols $  \myxi,\eta $  will denote elements in the Euclidean space written as row vectors. For ease of distinction,
we  denote the space of row vectors by $\rnhat$.
The transpose of a vector or matrix $x$ is denoted by $\trp{x}$, hence the inner product in $\R^n$ is
$x \cdot y  = \trp{y}  x  $. 
The \emph{Fourier transform} is given by
\[	
		\hat f ( \myxi ) = \int_{\R^n} f( x ) e^{-2i \pi  \myxi   x } \, d x 	
\]
for $f \in L^1(\R^n)$ and $ \myxi  \in \rnhat$. The restriction of the map $f \mapsto \hat f$ to $(L^1 \cap L^2)(\R^n)$
extends to a unitary operator $\mathcal F: f \in L^2(\R^n) \mapsto \hat f \in L^2(\rnhat)$ which is also called the Fourier transform.

\emph{Translation} and \emph{modulation} define two unitary representations of $\R^n$  on $L^2(\R^n)$  by
\[	(T_xf)(y)=f(y-x) \qquad \text{and} \qquad    (E_{x} f)(y) = e^{2i\pi \trp{x} y } f(y),	\]
and the corresponding operators on $L^2(\rnhat)$ are defined similarly,
\[	
			(\hat T_{x}g)(\myxi)=g(\myxi- \trp{x} ) \qquad \text{and} \qquad    (\hat E_{x} g)(\myxi) = e^{2i\pi \myxi x } g(\myxi),	
\]
for $x,y \in \R^n$, $\myxi\in \rnhat$, $f \in L^2(\R^n)$ and $g \in L^2(\rnhat)$.
The natural representations of $GL_n(\R)$ on these spaces are given by left and right \emph{dilation}, respectively,
\[	(\mydil_a f)(y) = | \det a|^{-1/2} f(a^{-1}y) \qquad \text{and} \qquad  (\hat \mydil_a g)(\myxi) = | \det a|^{1/2} g(\myxi a ), 	\]
for $a \in GL_n(\R)$.
Observe that 
\begin{equation}\label{eq;conj}
			\hat E_{-x}=\mathcal{F}  T_x  \mathcal{F}^{-1}, \quad
			\hat T_{x} = \mathcal{F} E_{x} \mathcal{F}^{-1} \quad \text{and} \quad \hat \mydil_a = \mathcal{F} \mydil_a \mathcal{F}^{-1}.
\end{equation}

By isomorphism of groups we will always mean an isomorphism of topological groups.

%%%%%%%%%%%%%%%%%%%%%%%%%%%%%%%%%%%%%%%%%%%%%%%%%%%%%%%%%%%%%%%%%%%%%%%%%%%%%%

\medskip\noindent
\textbf{The affine group and the wavelet representation.}
  The \emph{affine group} $\textit{Aff}(n,\R)$ is the group formed by the invertible linear transformations and translations in Euclidean
 space.  It takes the form of a semi-direct product $\R^n \rtimes_{\alpha} GL_n(\R)$, where the action $\alpha$ is simply
 matrix multiplication, $	\alpha_{a}(x) = ax$ for $x \in \R^n$, $a \in GL_n(\R)$. Thus the group
 operation is 
 \[	
 			(x,a) (\tilde x, \tilde a)=(x+a \tilde x, a \tilde a)	
 \]
for $(x,a),(\tilde x,\tilde a) \in \textit{Aff}(n,\R)$.  If $H$ is a closed subgroup of $GL_n(\R)$, then the corresponding
subgroup of the affine group can be represented as the matrix group
\[		
		 \R^n \rtimes_{\alpha} H \cong \left \{ \, \begin{bmatrix} a & x \\ 0 & 1 \end{bmatrix} \, : \,
		 				x \in \R^n, \ a \in H \, \right \} \subset GL_{n+1}(\R).
 \]
There is a natural unitary representation $\pi$ of such subgroups  on $L^2(\R^n)$, called the \emph{wavelet representation}
and determined by translations and left dilations,
\[		
			\qquad\qquad\pi(x,a) = T_x \mydil_a, \qquad (x,a) \in  \R^n \rtimes_{\alpha} H .
\]
Conjugating by the Fourier transform, (\ref{eq;conj}) yields an equivalent representation $\hat \pi$ on $L^2(\rnhat)$ given by
\begin{equation}\label{rep:wav2}
		\qquad \hat \pi(x,a) = \hat E_{-x} \hat \mydil_a, \qquad (x,a) \in  \R^n \rtimes_{\alpha} H .
\end{equation}

%%%%%%%%%%%%%%%%%%%%%%%%%%%%%%%%%%%%%%%%%%%%%%%%%%%%%%%%%%%%%%%%%%%%%%%%%%%%%%

\medskip\noindent
\textbf{The symplectic group and the metaplectic representation.} 
The  symplectic group and its representations has been extensively studied in \cite{Folland} and \cite{Grochenig}.
The matrix $\mathcal{J}$ is one of its elements, and the additive group $\textit{Sym}(n,\R)$ of symmetric $n \times n$ matrices 
as well as $GL_n(\R)$ are naturally embedded in $Sp(n,\R)$ in form of the closed subgroups
\begin{equation}\label{eq:NL}
\begin{split}
  	N = \left\{ \mathcal{N}_m := \begin{bmatrix}
  									 I_n & 0  \\
   									 m & I_n  \\
 								\end{bmatrix} 
 				 : \, m \in \textit{Sym}(n,\R)  \right \}, \quad
 	 				   %\quad \text{and} \\
      	L = \left\{ \mathcal{L}_a := \begin{bmatrix}
		   											a & 0  \\
			  									 	0 & \trp{ (a^{- 1}) }   \\
 											\end{bmatrix} 
 				 : \,  a \in GL_n(\R) \right\},
\end{split}
\end{equation}
and $Sp(n, \mathbb{R})$ is generated by $L \cup  N \cup \left\{ \mathcal{J} \right\}$.
There is a  projective representation $\mu$ of $Sp(n,\R)$ on $L^2(\R^n)$ called the
\emph{metaplectic representation} which, for the three types of generating matrices, is given by
\begin{align*}
 	 \mu  \left(\mathcal{L}_a \right)   =    \mydil_a, \qquad    		%% \label{eq302}  \\
	 \mu \left(\mathcal{N}_m  \right)       =   U_m, \qquad 		%%  \label{eq303} \\
	 \mu (-\mathcal{J}) = \left( { - i} \right)^{n/2} \mathcal{F}, %%\label{eq304}
\end{align*}
where  $U_m$ is a \emph{chirp},
\[
	\, (U_m f)( q ) = e^{i \pi \trp{q} m  q }f( q )  \, 
\]
for  $f \in L^2(\mathbb{R}^n)$ and $ q  \in \R^n$.

\medskip
\noindent
\textbf{Subgroups of the symplectic group which posses a wavelet representation.}
We next consider a class of subgroups of $Sp(n,\R)$ which arise as semidirect products
of a vector group with a group of dilations.
There is a natural linear action $\alpha$ of $GL_n(\R)$ on the vector space $\textit{Sym}(n,\R)$  given by
\begin{equation}\label{motiv:act}
	   \alpha_a(m) =  \trp{(a^{-1})} m a^{-1} \qquad\qquad  \big (a \in GL_n(\R), \ m \in Sym(n,\R)  \big ).	
\end{equation}
Let $E$ be a closed subgroup  $GL_n(\R)$ and $M$ an  $E$-invariant linear subspace  of $\textit{Sym}(n,\R)$.
As can be seen from (\ref{eq:NL}), $M$ and $E$ are isomorphic to closed subgroups of 
$Sp(n,\R)$, and the action $\alpha$ is implemented by conjugation under this isomorphism,
\[
	\mathcal L_a \mathcal N_m \mathcal L_a ^{-1} = \mathcal N_{ \trp{(a^{-1})} m a^{-1}}.
\]
Consequently, the semidirect product $M \rtimes_{\alpha} E$ is isomorphic to a closed subgroup of $Sp(n,\R)$,
\begin{equation}\label{eq:meta:mat}
		M \rtimes_{\alpha} E \cong K:= \left \{   \mathcal N_m \mathcal L_a = \begin{bmatrix}
   						a & 0  \\
  						 m a &  \trp{( a^{-1} )}
					\end{bmatrix}\ : \  m \in M, \ a \in  E \right   \}	.
\end{equation}
The restriction of the metaplectic representation to $K$, which we simply call the metaplectic representation
of $M \rtimes_{\alpha} E$, is  given by
\begin{equation}\label{eq3301}
		\qquad \mu(m,a) :=   \mu ( \mathcal N_m \mathcal L_a)  		 = U_m \mydil_a, \qquad\qquad ( \, (m,a) \in M \rtimes_{\alpha} E \, ),
\end{equation}
and it is a proper representation, that is, a group homomorphism.

\medskip
The groups $M \rtimes_{\alpha} E$ have a wavelet representation as well. In fact, identify $M$ with Euclidean space $\R^d$ by fixing a basis.
Since the action $\alpha$   is by invertible linear transformations, there exists a continuous homomorphism 
$\varphi : a \mapsto h_a$ of $E$ onto a (not necessarily closed)
subgroup $H$ of $GL_d(\R)$  satisfying
\[	\alpha_a(m) = h_a m \qquad (m \in \R^d, \ a \in E )	,\]
which naturally extends to a group homomorphism $\varphi$ of $M \rtimes_{\alpha} E$
onto the subgroup  $\R^d \rtimes_{\alpha} H$ of  $\textit{Aff}(d,\R)$ by
\begin{equation}\label{eq:phi}
			\	\varphi(m,a) = (m,h_a).	\ 
\end{equation}
%Note that $\varphi$ is an isomorphism if and only if the action of $E$ on $M$ is effective. 
(For ease of notation, we will denote these semi-direct products simply by $M \rtimes E$ and $\R^d \rtimes H$.)
Now composition of the homomorphism $\varphi$ with the wavelet representation (\ref{rep:wav2}) in Fourier space
 yields a wavelet representation of $M \rtimes E$ on $L^2(\widehat{\R^d})$, also denoted by $\hat\pi$, and given by
\begin{equation}\label{eq:metwav}
					\hat\pi( m,a) = \hat E_{-m} \hat \mydil_{h_a}.	
\end{equation}

%%%%%%%%%%%%%%%%%%%%%%%%%%%%%%%%%%%%%%%%%%%%%%%%%%%%%%%%%%%%%%%%%%%%%%%%%%%%%%
\subsection{The groups $G_{p,B}$ are subgroups of $Sp(n+1,\R)$ and $A\!f\!f\!(n+1,\R)$}

We now show that the each group $G_{p,B}$ can be represented as a subgroup of the form $M \rtimes E $ of the symplectic group, and $\mathbb{R}^{n+1} \rtimes H$ of the affine group as discussed above.
We will impose the assumptions (M1) and (M2) of the section 2, which ensures that each group $G_{p,B}$ can be represented as a matrix group
of the form (\ref{eq:8}).

From now on, $M$ will denote the $n+1$ dimensional vector subspace of $\textit{Sym}(n+1,\mathbb R)$,
\begin{equation}\label{eq:101}
M  = \left\{ {m(z, x): = \left[ {\begin{array}{*{20}c}
  		 -z & { -x^T }  \\
   		{ -x} & 0  \\
 \end{array} } \right]\,:\, x \in \mathbb{R}^n , \ z \in \mathbb{R}} \right\}
\end{equation}
and $E=E_{p,B}$, the closed subgroup of $GL_{n+1}(\mathbb R)$,
 \begin{equation*}%\label{group:dab}
	E_{p,B}  =  \left\{  a(t, y )
		: = \begin{bmatrix} 
					1  & 0  \\
			 		- \frac{1}{2}  y & I_n
			 \end{bmatrix} 
			  \begin{bmatrix} 
					e^{-pt/2}  & 0  \\
					 0 & e^{pt/2}  \left [e^{-Bt}\right ]^{\transp}
			 \end{bmatrix}
			\ : \ t \in \mathbb{R}^2,  \ y \in \R^n  \right\}.
\end{equation*}
The group law in $E_{p,B}$ is
\begin{equation}\label{sec5:lawd}
	a(t,y) a( \tilde t, \tilde y )=	a(t+\tilde  t,y + e^{pt/2}  \left [e^{-Bt}\right ]^{\transp} \tilde y) .	
\end{equation}
Now $M$ is invariant under the $E_{p,B}$-action (\ref{motiv:act}), in fact
\begin{equation}\label{eq:104}
		\alpha_{a(t,y)}(m(z,x))=(a(t,y)^{-1})^\transp m(z,x) a(t,y)^{-1} = m(e^{pt}z + y^\transp e^{Bt}x, e^{Bt}x).
\end{equation}
By (\ref{eq:meta:mat}), the semi-direct product $ M \rtimes E_{p,B}$ 
can be identified with a closed subgroup of $Sp(n+1,\R)$,
\[
	 M \rtimes  E_{p,B} \cong K_{p,B} 
	 			:= \left\{ k(t, x,  y,z)  %=\mathcal{N}_{m(z,x)} \mathcal{L}_{a(t, y)} 
	 			= \begin{bmatrix}
   					a(t, y) & 0  \\
			   		 m(z, x)a(t, y) \ & \left( a(t, y) ^{-1} \right )^\transp   
				 \end{bmatrix}\ :\ z \in \mathbb{R},\, t \in \R^2, \,  x,  y \in \mathbb{R}^n  \right \}
\]
with the group law
\begin{equation*}%\label{equa:08}
		k(t,  x, y, z) \, k(\tilde t,  \tilde x, \tilde y, \tilde z)
				= k(t + \tilde t, x + e^{Bt} \tilde x, y + e^{pt} \left [ e^{- B   t}\right ]^{\transp}\tilde y,z 
								+ e^{pt} \tilde z + y^T e^{Bt}\tilde x)
\end{equation*}
which is the same as the  group law (\ref{eq:4})  of $G_{p,B}$. It is now easy to see that the matrix groups
$G_{p,B}$ and $K_{p,B}$ are isomorphic.

%On the other hand, since $E_{p,B}$ acts effectively on $M$, then $M \rtimes E_{p,B}$ is 
% also isomorphic to a subgroup %
Next we compute the homomorphism $\varphi :  M \rtimes E_{p,B} \to \mathbb{R}^{n+1} \rtimes H$ of (\ref{eq:phi}).
By (\ref{eq:101}), the vector space $M$  is naturally identified with $\R^{n+1}$
via the map $ m(z,x) \mapsto \left ( \begin{smallmatrix} z \\ x \end{smallmatrix} \right )$.
Equation (\ref{eq:104}) now shows that under this identification,
\[	
			h_{a(t,y)} \begin{pmatrix} z \\ x \end{pmatrix}  
			= \begin{pmatrix}  e^{pt} z + y^{\transp} e^{Bt} x \\[2pt] e^{Bt}x \end{pmatrix},
\]
so that
\[		
	H = H_{p,B} = \left \{ 	h_{a(t,y)}  = \begin{bmatrix} e^{pt} & y^{\transp} e^{Bt} \\ 0 & e^{Bt} \end{bmatrix}
		: \, t \in \R^2, \ y \in \R^n \, \right \}.
\]
We observe that by assumptions (M1)--(M2), this group is closed in $GL_{n+1}(\R)$, and the map $ \varphi : E_{p,B} \to H_{p,B}$
is an isomorphism of matrix groups. Hence,
\begin{align*}
	 G_{p,B} \cong M \rtimes E_{p,B} \cong
	\mathbb{R}^{n+1} \rtimes H_{p,B}   &= \left \{ (m,h_a) : m \in \R^{n+1}, \, h_a \in H_{p,B} \right \} \\[4pt]
				&\cong  \left \{ \, \begin{bmatrix}
   							h_{a(t,y)} & \left ( \begin{smallmatrix} z \\  x  \end{smallmatrix} \right ) \\[2pt]
  										 0 & 1  \\
 								\end{bmatrix} 	\ : \
 			z \in \R, \,  t \in \R^2, \, x,y \in \R^n \, \right\} .
\end{align*}
which is a closed subgroup of $\textit{Aff}(n+1,\R)$.

%%%-----------------------------------------------------

\subsection{The symplectic and wavelet representations of the groups $G_{p,B}$}

%-------------------------------------------------

By (\ref{eq:metwav}), the wavelet representation of $G_{p,B} \cong \R^{n+1} \rtimes H_{p,B}$ in Fourier space is given by
\[			\hat \pi \bigl ( \, g( t, x, y, z )  \, \bigr )
		 = \hat E_{- \left (\begin{smallmatrix} z \\ x \end{smallmatrix} \right )} \hat \mydil_{h_{a(t,y)}},
\]
that is,
\begin{equation}\label{eq:wav:2}
	\left [ \hat \pi \bigl ( g ( t, x, y, z ) \bigr )f \right]( r , \myxi ) = \delta(t)^{1/2} \, e^{pt/2} 
			e^{-2i\pi( rz + \myxi x )} f \left ( r e^{pt} , ( r \trp{y} + \myxi ) e^{Bt} \right )
\end{equation}
for $f \in L^2( \widehat{\R^{n+1}})$, $r \in \R$, $\myxi  \in \rnhat$ and 
$\delta(t) = \det\left ( {e^{Bt}} \right )=e^{\textit{tr}(Bt)}$.
Clearly, $\widehat{ \R^{n+1}}$ decomposes measurably into the two $H_{p,B}$-invariant open half spaces
\[	
	\mathcal{O}_+ = \{ (r,\myxi ) : r>0 \} \qquad \text{and} \qquad \mathcal{O}_- = \{ (r,\myxi ) : r<0 \} .
\]
It thus can be seen from (\ref{eq:wav:2}) that  $L^2(\mathcal{O}_{+})$ and $L^2(\mathcal{O}_{-})$ are both 
$\hat \pi$-invariant subspaces of $L^2( \widehat{\R^{n+1}})$ and
consequently, the wavelet representation $\hat \pi$ splits into the direct sum $\hat \pi = \hat \pi_+ \oplus \hat \pi_{-}$
of the two subrepresentations $\hat \pi_{\pm}$
obtained by restricting $\hat \pi$ to these two invariant subspaces.

Similarly, by (\ref{eq3301}),  the metaplectic representation of the group $G_{p,B} \cong  M \rtimes E_{p,B}$ 
is given by 
\[		\mu\left (\,  g(t, x ,y,z)\,  \right ) = U_{m(z, x )} \mydil_{a(t,y)} .	\]
Since for each  vector $q=\left  ( \begin{smallmatrix} u \\ v \end{smallmatrix} \right ) \in \R^{n+1}$, $u \in \R$, $v \in \R^n$ we have
\begin{equation}\label{eq:83}	
		\trp{q} m(z,x) q = (u, \trp{v})  m(z,  x )
				 \left ( \begin{smallmatrix} u \\ v \end{smallmatrix} \right ) 	= -(u^2 z+  2u  \trp{v} x ) ,	
\end{equation}
it follows that
\begin{equation}\label{group:meta1}
\begin{split}
		\mu\bigl  (\,  g(t, x ,y,z)\,  \bigr ) f  \left ( \begin{matrix} u \\ v \end{matrix} \right )
			= \delta(t)^{1/2}  e^{pt(1-n)/4} e^{- i \pi (u^2 z+ 2 u  \trp{v} x )} 
					f\begin{pmatrix} e^{pt/2}u \\ e^{-pt/2} \trp{\left [e^{Bt} \right] } ( \frac{u}{2} y +  v ) \end{pmatrix}
\end{split}
\end{equation}
 for $f \in L^2(\R^{n+1})$,  with  $\delta(t) = \det{\left (e^{Bt}\right )}$.
Clearly,  $\R^{n+1}$ splits measurably into two $E_{p,B}$-invariant open half spaces
\[	
		\mathcal{U}_+  = \{ \left ( \begin{smallmatrix} u \\ v \end{smallmatrix} \right ) \ : \ u > 0 \} \quad   \text{and} \quad
		\mathcal{U}_{-} = \{ \left ( \begin{smallmatrix} u \\ v \end{smallmatrix} \right ) \ : \ u< 0 \}.	
\]
It can be seen from (\ref{group:meta1}) that $L^2(\mathcal{U}_{+})$ and $L^2(\mathcal{U}_{-})$ are both $\mu$-invariant subspaces of
 $L^2( \R^{n+1})$. Hence,   $\mu$ splits into the direct sum $\mu = \mu_+ \oplus \mu_{-}$
of the two subrepresentations $\mu_{\pm}$ obtained by restricting $\mu$ to each of the two invariant subspaces $L^2(\mathcal{U}_{\pm})$.
		
We next obtain a connection  between these representations, by employing the the techniques developed in 
\cite{Cordero,DMari,Namngam2}.

\begin{prop}
The subrepresentations $\mu_+$ and $\mu_-$ are both equivalent to $\hat \pi_+$.
\end{prop}
\begin{proof}
Observe that for each $q  \in \R^{n+1}$, the map   
\[		 \begin{pmatrix} z \\ x \end{pmatrix} 	  \mapsto q^T m(z,x)  q 
\] 
defines a linear functional on $\R^{n+1}$. Hence there exists a unique $\Psi(q) \in \widehat{\R^{n+1}}$ so that
\[				
		q^T m(z,x)  q  = - 2 \Psi(q) \begin{pmatrix} z \\ x \end{pmatrix} 	\]
for all $z \in \R$, $x \in \R^n$. In fact, equation (\ref{eq:83}) shows that
\[	
			\Psi(q) = \Psi \!  \begin{pmatrix} u \\ v \end{pmatrix} 
					=  \Bigl ( \textstyle \frac{1}{2} u^2, u \trp{v} \Bigr ).	
\]
We observe that $\Psi$ is smooth with Jacobian determinant
\[			
		 J_{\Psi}\! \begin{pmatrix} u \\ v \end{pmatrix}    = u^{n+1} 
\]
which does not vanish on the open half planes $\mathcal{U}_+$ and 
$\mathcal{U}_{-}$. In fact, the restrictions of $\Psi$
to these sets constitute diffeomorphisms
\[		
	\Psi_+ : \mathcal{U}_+ \to \mathcal{O}_+  \qquad \text{and} \qquad \Psi_{-} : \mathcal{U}_- \to \mathcal{O}_+ ,
\]
respectively. Furthermore, for $(r, \xi) \in \mathcal{O}_+ \subset \widehat{\R^{n+1}}$ with $r \in \R$, $\xi \in \rnhat$ we have
\[			
		 \Psi_{\pm}^{-1} (r, \xi) = 
		 	 \begin{pmatrix} \pm \sqrt{2r} \\ \pm \frac{1}{\sqrt{2r}}  \, \trp{\xi} \end{pmatrix}  
		 	 \qquad \text{and} \qquad
		 	 J_{\Psi_{\pm}^{-1}}(r,\xi) =\pm  \left (  2r \right )^{-(n+1)/2} .
\]
It follows that the operators
\[
	Q_+: L^2(\mathcal{O}_+) \to L^2(\mathcal{U}_+) \qquad \text{and} \qquad Q_-: L^2(\mathcal{O}_+) \to L^2(\mathcal{U}_-) 
\]
defined by
\[
	\qquad\qquad 	(Q_{\pm}f)(q) = \left | J_{\Psi}(q) \right |^{1/2} f \left ( \Psi(q) \right ) \qquad\qquad
				( f \in  L^2(\mathcal{O}_+),  q \in  \mathcal{U}_\pm )
\]	
constitute Hilbert space isomorphism, whose inverses are given by
\[
	\qquad\qquad 
		(Q_{\pm}^{-1}f)(\eta) = \left | J_{\Psi_{\pm}^{-1}}(\eta) \right | ^{1/2} 
				f \left ( \Psi_{\pm}^{-1}(\eta) \right ) \qquad\qquad
					( f \in  L^2(\mathcal{U}_\pm),  \eta \in  \mathcal{O}_+ ).
\]	
We complete the proof by showing that
\[			\mu_{\pm} = Q_{\pm} \hat \pi_+ Q_{\pm}^{-1}.	\]
In fact, for all $f \in L^2(\mathcal{U}_{\pm})$ and $q= \left ( \begin{smallmatrix} u \\ v \end{smallmatrix} \right ) \in \R^{n+1}$
we have
\begin{align*}
	\Bigl [  Q_{\pm}  &\hat \pi_+(t,x,y,z)  Q_{\pm}^{-1}	f \Bigr ] (q)
		= \left | J_{\Psi} \! \left ( \begin{smallmatrix} u \\ v \end{smallmatrix} \right )  \right |^{1/2}
		 \left [  \hat \pi_+(t,x,y,z)  Q_{\pm}^{-1}	f \right ]  
						\left (  \Psi \left ( \begin{smallmatrix} u \\ v \end{smallmatrix} \right )  \right )  \\
		&= |u|^{(n+1)/2} \delta(t)^{1/2} e^{pt/2} e^{-2i \pi \left ((u^2/2) z + u \trp{v}x \right ) } 
				\left [  Q_{\pm}^{-1}	f  \right ] \! 
					\left ( \textstyle \frac{u^2}{2} e^{pt}, \bigl ( \frac{u^2}{2}\trp{y} + u\trp{v} \bigr ) e^{Bt} \right ) \\
		&= \delta(t)^{1/2} |u|^{(n+1)/2} e^{pt/2} e^{-i \pi \left (u^2 z + 2 u\trp{v}x \right ) } 
				\left | \pm u^2 e^{pt} \right |^{-(n+1)/4} f \! 
						\left ( \begin{matrix} \pm \sqrt{u^2 e^{pt}} \\
										\pm \frac{ 1}{\sqrt{ u^2 e^{pt}}	}
										 \trp{[ e^{Bt}]} \left (\frac{u^2}{2}y+ uv \right )
								\end{matrix} \right ) \\
		&= \delta(t)^{1/2}   e^{pt(1-n)/4} e^{-i \pi \left (u^2 z + 2 u\trp{v}x \right ) } 
						 f \! 
						\left ( \begin{matrix} e^{pt/2} u \\
										 e^{-pt/2}	 \trp{[ e^{Bt}]} \left (\frac{u}{2}y+v \right )
								\end{matrix} \right ) 
\end{align*}
which is precisely (\ref{group:meta1}).
\end{proof}

It now follows immediately that the metaplectic representation $\mu$ of $G_{p,B}$  
 is equivalent to the sum of two copies of $\hat \pi_+$,
\[		\mu = \mu_+ \oplus \mu_{-} \simeq \hat \pi_+ \oplus \hat \pi_+.	\]

%%%%%%%%%%%%%%%%%%%%%%%%%%

\bigskip

\noindent{\bf Acknowledgements} A.S. is grateful to the Development and Promotion of Science and Technology Talents project (DPST) for their constantly support.

%%%%%%%%%%%%%%%%%%%%%%%%%%

%%%%%%%%%%%%%%%%%%%%%%%%%%

\end{document}